\documentclass{amsart}

\usepackage{amsmath}
\usepackage{amsthm}
\usepackage{a4wide}
\usepackage{enumerate}

\newcommand{\R}{\mathbf{R}}
\newcommand{\Z}{\mathbf{Z}}
\newcommand{\OX}{\mathcal{O}_X}
\newcommand{\Ok}{\mathcal{O}_k}
\newcommand{\Ahat}{\widehat{A}}
\newcommand{\Bhat}{\widehat{B}}
\newcommand{\Chat}{\widehat{C}}
\newcommand{\Rhat}{\widehat{R}}
\newcommand{\Shat}{\widehat{S}}
\newcommand{\Uhat}{\widehat{U}}

\DeclareMathOperator{\Spa}{Spa}

\theoremstyle{plain}
\newtheorem{theorem}{Theorem}
\newtheorem{lemma}[theorem]{Lemma}
\newtheorem{corollary}[theorem]{Corollary}
\newtheorem{proposition}[theorem]{Proposition}
\theoremstyle{remark}
\newtheorem{remark}[theorem]{Remark}
\newtheorem*{remarkn}{Remark}

\begin{document}

\title{Stably uniform affinoids are sheafy}

\begin{abstract}
We develop some of the foundations of affinoid pre-adic spaces without
Noetherian or finiteness hypotheses. We give explicit examples of
non-adic affinoid pre-adic spaces, and also a new condition ensuring that
the structure presheaf on $\Spa(R,R^+)$ is a sheaf. This condition can
be used to give a new proof that the spectrum of a perfectoid algebra
is an adic space.
\end{abstract}

\author{Kevin Buzzard}\email{k.buzzard@imperial.ac.uk}\address{Department of
Mathematics, Imperial College London}
\author{Alain Verberkmoes}
\maketitle

\thispagestyle{empty}

\section{Introduction}

Let $k$ be a field complete with respect to a non-trivial
non-archimedean norm $k\to\R_{\geq0}$. If $(R,R^+)$ is an affinoid $k$-algebra
(in the sense of Definition~2.6(ii) of~\cite{scholze:perfectoid}) then we can
associate to it a certain topological space $X:=\Spa(R,R^+)$ whose elements
are certain valuations on $R$. This topological space -- a so-called
affinoid pre-adic space -- has a natural
presheaf of complete topological rings $\OX$ on it. The presheaf is known to be a sheaf
if $R$ satisfies certain finiteness conditions. For example, it
is a sheaf if $R$ is a quotient of
a Tate algebra $k\langle T_1,T_2,\ldots,T_n\rangle$ (that is, the ring
of of power series which converge on the closed unit polydisc), and
there are other finiteness conditions which also suffice to guarantee
$\OX$ is a sheaf. 
These finiteness conditions are imposed very early on
in~\cite{huber:book} (Assumption~(1.1.1)), which is mainly concerned
with the theory of \'etale cohomology in the context of rigid spaces.
However, more recently Scholze has introduced the concept of a
perfectoid $k$-algebra, for which these finiteness conditions essentially
never hold. Scholze showed in Theorem~6.3(iii) of~\cite{scholze:perfectoid}
that for $R$ perfectoid, $\OX$ is still
a sheaf. His proof is delicate, involving a direct calculation
in characteristic~$p$ and then some machinery (almost mathematics,
tilting) to deduce the result in characteristic zero.
However it would be technically useful in some applications
to have a more general
method. For example, in Conjecture~2.16 of~\cite{scholze:survey} Scholze
asks the following question: if $X$ can be covered by rational subsets
which are perfectoid, then is $X$ perfectoid? This unfortunately
turns out not to be true in this generality, because there are examples
of locally perfectoid affinoid pre-adic spaces $X$ where $\OX$ fails to be
a sheaf and hence $X$ cannot be perfectoid. This raises the general question
of what extra assumptions one should put on~$X$ in order to hope
that one can check that it is perfectoid via local calculations.
However, except for an example by Mihara (\cite{mihara}) posted to the
ArXiv whilst this paper was being typed up, there seem to be no examples in
the literature at all of affinoid $k$-algebras for which $\OX$ is not a sheaf and
in general the problem seems to be very poorly-understood (or at least
poorly-documented).

In this paper we give some examples of affinoid $k$-algebras for
which $\OX$ is not a sheaf, and show that the phenomenon is strongly
linked to the issue that the set of power-bounded elements in an affinoid
ring may not be bounded. On the other hand, we show that if every rational
subset of $\Spa(R,R^+)$ has the property that all power-bounded elements
are bounded, then $\OX$ is a sheaf (with no finiteness or perfectoid
assumptions). Scholze has used this latter result to verify sheafiness
for certain constructions underlying his new theory of diamonds.

Note finally that if one drops the finiteness conditions that Huber imposes
then one might not expect a reasonable theory of coherent sheaves; this
is an issue even in the perfectoid space setting. For example
it seems to be currently unclear whether an open immersion of affinoids
induces a flat morphism on rings of global sections in this generality.

\subsection*{Acknowledgements.}

KB would like to thank Torsten Wedhorn for his notes on adic spaces,
which he found a very useful introduction to the subject,
Peter Scholze for encouragement and guiding comments, and
AV for inviting him to Hakkasan, which began this collaboration.
Both authors would like to thank Brian Conrad for his detailed
comments on the manuscript, and also for urging the authors to write
the argument in the general setting of Tate rings.

\section{Definitions}

We recall from section~1 of~\cite{huber:valuations} that a {\em Tate ring} 
is an f-adic ring that has a topologically
nilpotent unit.
More concretely, it is a topological ring that can be
obtained as follows.  Let $R$ be a ring, $R_0$ a sub-ring of $R$, and
$\varpi\in R_0\cap R^\times$ such that $R=R_0[\varpi^{-1}]$. The subsets
$r+\varpi^nR_0$ with $r\in R$ and
$n\geq0$ form the basis of a topology on
$R$ and the resulting topological ring is a Tate ring (see \cite{huber:valuations}
Proposition~1.5). Note that without the condition $R=R_0[\varpi^{-1}]$,
multiplication in $R$ would not be continuous.
We will refer to this topology on $R$ as the topology induced by the
subring $R_0$ and the ideal $(\varpi)$ of $R_0$.
In general, there are different
choices for $R_0$ and $\varpi$ that lead to the same topology on~$R$.

Here is a construction of Tate rings which will be used in the later
stages of this paper. Let $k$ be a field complete with
respect to a non-trivial non-archimedean norm $|.|:k\to\R_{\geq0}$, let~$\Ok$
be its integer ring, and let~$\varpi$ be any element of $k^\times$ with
$0<|\varpi|<1$. A \emph{Tate $k$-algebra} (as in Definition~2.6 of~\cite{scholze:perfectoid}) is a $k$-algebra~$R$ equipped with an $\Ok$-subalgebra~$R_0$
such that $kR_0=R$. All the explicit examples of Tate rings appearing
in section~\ref{counterexamples} in this paper are Tate $k$-algebras.

An element $r$ of a Tate ring $R$ is called \emph{power-bounded} if there
exists some $n\geq0$ such that $r^m\in\varpi^{-n}R_0$ for all $m\geq0$
(this property depends only on the topology on~$R$ rather
than the explicit choice of $R_0$ and~$\varpi$).
The set~$R^\circ$
of power-bounded elements is an open and integrally closed subring of~$R$
containing $R_0$. If $R^+\subseteq R^\circ$ now denotes an arbitrary open
and integrally closed 
subring of~$R$ (for example, $R^+=R^\circ$)
then the resulting pair $(R,R^+)$ is called a \emph{Tate affinoid ring},
and to this pair one can associate a topological space $X=\Spa(R,R^+)$
whose elements are (equivalence classes of) continuous valuations on~$R$
which are bounded by~1 on~$R^+$. The space~$X$
is furthermore endowed with a presheaf $\OX$ of complete topological rings,
and the question this paper is mainly concerned with, is when this presheaf
is a sheaf.

The first sections of~\cite{huber:valuations} and~\cite{huber:generalization}
give careful definitions of~$X$ and $\OX$; other good references
are~\cite{wedhorn} and
(for the case of $k$-algebras)~\cite{scholze:perfectoid}.
We summarize here the facts that we will need. Firstly we
mention some basic results about completions for which we could find
no easily-accessible reference. Let $R$ be a Tate ring, with $R_0$ and $\varpi$
as above. We can complete~$R$ with
respect to the topology defined by $R_0$ and $(\varpi)$; the
completion, denoted $\Rhat$, is the limit $\varprojlim_nR/\varpi^nR_0$,
endowed with the projective limit topology (the quotients $R/\varpi^nR_0$
have the discrete topology). There is a canonical map $i:R\to\Rhat$.
If $U\subseteq R$ is an open subgroup of the group $(R,+)$
then let $\Uhat$ denote the closure
of $i(U)$ in $\Rhat$ (this is just notation, but it is reasonable
because the closure
of $i(U)$ is isomorphic to the completion of~$U$ in the
sense of~\cite{bourbaki:topgen}, by
\cite{bourbaki:topgen} II,
\S3.9 Corollaire 1). 

\begin{lemma}\label{basic}
In the situation described above:
\begin{enumerate}[(i)]
\item The subgroups $\widehat{\varpi^n R_0}=\varpi^n\widehat{R_0}$ form a basis
of open neighbourhoods of the origin of $\Rhat$.
\item If $U$ is an open subgroup of $R$ then $i^{-1}(\Uhat)=U$.
\item $\Rhat$ naturally has the structure of a Tate ring, with topology
induced from the subring $\widehat{R_0}$ and element $i(\varpi)$;
moreover $\widehat{R^\circ}=(\Rhat)^\circ$ (i.e., completion commutes with taking
power-bounded elements).
\end{enumerate}
\end{lemma}

\begin{proof}
(i)
The closure~$\overline{\{0\}}$ of $\{0\}\subseteq R$
is easily checked to be $\cap_n\varpi^nR_0$, so the result
follows by applying \cite{bourbaki:topgen} III \S3.4 Proposition~7
to $R$ modulo this closure, noting that any subgroup containing
an open subgroup is open.

(ii) Clearly $i^{-1}(\Uhat)$ contains $U$. Conversely,
the universal property of the completion (\cite{bourbaki:topgen} III, \S3.4,
Proposition~8)
gives us a group homomorphism $\Rhat\to R/U$ through which the canonical
map $R\to R/U$
factors. The kernel $K$ of $\Rhat\to R/U$ has the property that it
contains $i(U)$ and hence its closure, but also that $i^{-1}(K)=U$. This
shows $i^{-1}(\Uhat)=U$. 

(iii) $\Rhat$ and $\widehat{R_0}$ are rings by
\cite{bourbaki:topgen} III, \S6.5 Proposition~6 and II, \S3.9 Corollaire 1.
The topology on $\Rhat$ is induced from $\widehat{R_0}$ by
\cite{bourbaki:topgen} III \S3.4 Proposition~7. By part (i) $i(R^\circ)$
consists of power-bounded elements, and $(\Rhat)^\circ$ is open
so it contains $\widehat{R^\circ}$. Conversely, say $\hat{r}\in\Rhat$
is power-bounded.
Because $i(R)$ is dense in $\Rhat$ and $(\Rhat)^\circ$ is open,
for any $n\geq0$ we may find $r_n\in R$ such that
$\hat{r}-i(r_n)\in\varpi^n\widehat{R_0}$, and one checks
using the binomial theorem that $i(r_n)$ is power-bounded
and hence (using (ii)) that $r_n$ is too. Hence $\hat{r}\in\widehat{R^\circ}$.
\end{proof}

We now return to our description of $\OX$. The ring $\OX(X)$ is not~$R$,
but the completion $\Rhat$ of $R$ with respect
to the topology induced by $R_0$ and $(\varpi)$.
Let us now describe $\OX$ on certain open subsets of $X$.
Choose $t\in R$. Then we can cover
$X$ by two open subsets $U:=\{x:|t(x)|\leq1\}$ and $V:=\{x:|t(x)|\geq1\}$
(where we make the standard abuse of notation: $x$ is a valuation on~$R$
and $|t(x)|$ is just another way of writing $x(t)$).
If the (still to be defined presheaf) $\OX$ is a sheaf of complete topological
rings then in particular it is
a sheaf of abelian groups and so the sequence
\begin{equation}
0\to\OX(X)\to\OX(U)\oplus\OX(V)\to\OX(U\cap V)\tag{$*$}
\end{equation}
of abelian groups is exact.
We will now describe these 
groups and homomorphims explicitly. The subsets $U$, $V$ and $U\cap V$
are rational subsets of $X$, so it is not hard to compute $\OX$
on them directly.

Set $A=R$ and $A_0=R_0[t]$. We topologize the ring~$A$ 
using $A_0$ and $\varpi A_0$ as above.  The space $\OX(U)$ is the
completion $\Ahat$ of $A$ with respect to this topology.
Set $B=R[1/t]$, the localization of~$R$ at the set
$\{1,t,t^2,t^3,\ldots\}$ obtained by inverting~$t$.
Let $\phi:R\to B$ denote the canonical map.
We set $B_0=\phi(R_0)[1/t]$ and topologize~$B$ using $B_0$ and $\varpi B_0$
as above. The
space $\OX(V)$ is the completion $\Bhat$ of~$B$.
Finally we set $C=B$ and $C_0=\phi(R_0)[t,1/t]$.
The space $\OX(U\cap V)$ is the completion $\Chat$ of~$C$.

The abstract rings $R$ and~$A$ coincide, but their topologies will not coincide
in general. More precisely, $R_0\subseteq A_0$ and hence the
identity map $R\to A$ 
is continuous, but if $A_0\not\subseteq\varpi^{-n}R_0$ for every $n\geq0$
then the identity map $A\to R$ is not.
(An example where this happens is when $k$ is a field complete with respect
to a non-trivial non-archimedean valuation, $\varpi\in k$ with $0<|\varpi|<1$,
and $R=A=k[T]$, $R_0=\Ok[\varpi T]$ and $t=T$ so $A_0=\Ok[T]$.)
Similarly, the identity map $B\to C$ is continuous but the identity
map $C\to B$ may not be. Also, $\phi:R\to B$ is continuous as
are the induced maps $\phi:R\to C$ and $\phi:A\to C$, but the induced
map $A\to B$ may not be.

Define $\epsilon:R\to A\oplus B$ by $\epsilon(r)=(r,\phi(r))$, and
define $\delta:A\oplus B\to C$ by $\delta(a,b)=b-\phi(a)$. One
checks easily that the sequence of abstract abelian groups
\begin{equation}
0\to R\xrightarrow{\epsilon}A\oplus B\xrightarrow{\delta}C\to 0\tag{$**$}
\end{equation}
is exact, and indeed it is naturally split, the map $C\to A\oplus B$
sending $c$ to $(0,c)$ being a splitting. However if we topologize
$R$, $A\oplus B$ and $C$ using $R_0$, $A_0\oplus B_0$ and $C_0$
respectively, then $\epsilon$ and $\delta$ are continuous
but the splitting may not be continuous.

The sequence $(*)$ whose exactness we care about consists of the
first three arrows in the completion of the sequence $(**)$
with respect to the topologies defined by $R_0$, $A_0\oplus B_0$
and $C_0$. The issue then, is whether taking
completions can destroy left exactness.

Before we embark on a discussion of this, we recall the notion of strictness.
A continuous map between topological
groups $\psi:V\to W$ is called \emph{strict} if the two topologies on 
$\psi(V)$, namely the quotient topology coming from $V$ and the subspace
topology coming from $W$, coincide.
We see that $\delta:A\oplus B\to C$
is strict, because it is a continuous surjection and the image of
$A_0\oplus B_0$
is $C_0$ so $\delta$ is open. On the other hand, $\epsilon$ is strict iff
$S_0:=A_0\cap\phi^{-1}(B_0)$ is bounded in $R$, which is not always the
case; we will see explicit examples of that later on.

The following lemma shows that exactness of $(*)$ is in fact
equivalent to strictness of~$\epsilon$.

\begin{lemma}\label{exact}
The following are equivalent:
\begin{enumerate}[(i)]
\item $(*)$ is exact,
\item $(*)$ is exact and furthermore the map $\OX(U)\oplus\OX(V)\to\OX(U\cap V)$
is surjective,
\item there exists some $n\geq0$ such that $\varpi^n(A_0\cap\phi^{-1}(B_0))\subseteq R_0$,
\item $\epsilon$ is strict (and hence all maps in $(**)$ are strict).
\end{enumerate}
\end{lemma}

\begin{proof} 
Let $S$ denote the ring~$R$ and define $S_0:=A_0\cap\phi^{-1}(B_0)$.
Topologize $S$ using $S_0$ in the usual way (note $\varpi\in S_0$).
Then $\epsilon:S\to A\oplus B$
is strict and the identity map $R\to S$ is a continuous bijection.
In particular, strictness of $\epsilon$ is equivalent to $R\to S$
being a homeomorphism, which is equivalent to $R_0$ being
open in~$S$. Hence (iii) and (iv) are equivalent.

Now (ii) implies (i) trivially. Furthermore, 
it is a general fact in this setting that for an exact sequence
with all morphisms strict, its completion
remains exact (see for example~\cite{bourbaki:commalg} III.2.12, Lemme 2,
or~\cite{bgr} Corollary 1.1.9/6 for the case of $k$-algebras).
Applying this to the strict surjection $\delta$ we deduce that
$\OX(U)\oplus\OX(V)\to\OX(U\cap V)$ is always surjective (in fact
this is not difficult to see directly),
and so (i) implies~(ii). Furthermore, if (iv) holds then every
map in $(**)$ is strict, so the completion of $(**)$ is still
exact, and hence (iv) implies~(ii).

It suffices to prove that (i) implies (iii). Note
that the converse of the ``strict implies completion exact'' result used
several times above is not true in general.
(For example, if $R=A=k[T]$, $R_0=\Ok[\varpi T]$ and $A_0=\Ok[T]$,
then $R\to A$ is injective and not strict, but
the induced map $\Rhat\to\Ahat$ is still injective.)
Clearly $$0\to S\to A\oplus B\to C\to 0$$
is exact and all the maps are strict, so the sequence remains exact
under completion, and if furthermore (i) holds we deduce that the map
$\Rhat\to\Shat$ induced by the continuous map $R\to S$ must be a bijection.
We now wish to invoke the open mapping theorem in this
generality. One can check that the argument
of~\cite{bourbaki:topvs} Chapter~I, \S3.3, Th\'eor\`eme~1 holds in
this slightly more general setting of Tate rings. Another reference
is~\cite{henkel:omt}, and in the $k$-algebra case
there is~\cite{bgr} \S2.8.1.
As a consequence
we deduce that $\Rhat\to\Shat$ is open. In particular the image of $\Rhat_0$
must
contain $\varpi^n\Shat_0$ for some $n\geq0$. Pulling back via the
natural map $R\to\Rhat$ and using Lemma~\ref{basic}(ii)
we conclude that $R_0$ must contain $\varpi^nS_0$, which is (iii).
\end{proof}

This lemma is used in two ways in the sequel. In the next section
we observe that if the power-bounded elements of~$R$ are bounded,
then condition (iii) of the lemma follows
(Corollary~\ref{uniformimpliesstrict}),
and hence we get
a criterion for checking the sheaf axiom for the cover $X=U\cup V$,
which we can turn (Theorem~\ref{stablyuniformimpliessheaf}) into
a criterion for checking that the presheaf $\OX$ on a Tate affinoid
pre-adic space is a sheaf. As a consequence (Corollary~\ref{perfectoid})
we get a new proof that $\OX$ is a sheaf if $X=\Spa(R)$ with $R$ perfectoid.

In Section~\ref{counterexamples} we construct rings where
part (iii) is violated,
and use them to build examples of Tate affinoid pre-adic spaces which are
not adic.

\section{A criterion for $\OX$ to be a sheaf on a 
Tate affinoid pre-adic space}

Let~$R$, $R_0$ and $\varpi$ be as before.
As usual we topologize~$R$ by letting $\varpi^nR_0$
for $n\geq0$ be a basis of open neighbourhoods of zero.
We recall that $R^\circ$ denotes the subring of
power-bounded elements of~$R$. The ring~$R$
is called \emph{uniform} if $R^\circ$ is bounded, in other words if
there exists some $n\geq0$ such that $R^\circ\subseteq\varpi^{-n}R_0$.
Examples of uniform rings include reduced affinoid algebras in
Tate's original sense (i.e., those which are topologically of finite
type over a field~$k$), and conversely any Tate $k$-algebra with a non-zero
nilpotent element $r$ such that $kr\not\subseteq R_0$ would be
a non-uniform ring, as $kr\subseteq R^\circ$.

The key lemma we need in this section is that if an element of $R$ is
locally in $R_0$ then it is globally power-bounded. This sounds
geometrically reasonable, and we now give an
elementary
algebraic proof. We first remind the reader that every open cover
of an affinoid pre-adic space
can be refined to a rational cover (see Lemma~\ref{covers}(i) below,
and just before that lemma for the definition of a rational cover).

\begin{lemma}\label{alain}
Let $R$ be a Tate ring, with $R_0$ and $\varpi$ as before.
Let $t_1, \dots, t_n$ in $R$ such that $t_1 R+\dots+t_n R=R$.
For each $i$, let $R[1/t_i]$ be the localization of $R$ at the multiplicative set
$\{1, t_i, t_i^2, \dots\}$ and $\phi_i: R \to R[1/t_i]$ the natural homomorphism.
Then
$$\bigcap_{i=1}^n \phi_i^{-1}\bigl(\phi_i(R_0)[t_1/t_i, \dots, t_n/t_i]\bigr) \subseteq R^\circ.$$
\end{lemma}

\begin{proof}
Suppose $r \in \cap_{i=1}^n \phi_i^{-1}(\phi_i(R_0)[t_1/t_i, \dots, t_n/t_i])$.
For each $i$ there is a homogeneous polynomial $f_i \in R_0[T_1, \dots, T_n]$
such that $\phi_i(r) = t_i^{-\deg(f_i)} \phi_i(f_i(t_1, \dots, t_n))$.
Since $t_i^{\deg(f_i)} r - f_i(t_1, \dots, t_n) \in \ker(\phi_i)$, there exists
$c_i \geq 0$ such that $t_i^{c_i} (t_i^{\deg(f_i)} r - f_i(t_1, \dots, t_n)) = 0$.
So $t_i^{d_i} r = g_i(t_1, \dots, t_n)$ where
$g_i = T_i^{c_i} f_i \in R_0[T_1, \dots, T_n]$ is homogeneous of degree
$d_i=c_i+\deg(f_i)$.

Set $N=d_1+\dots+d_n$.
Take $A\geq0$ such that $\varpi^At_i\in R_0$ for all~$i$.
We will show, by induction on $m\geq0$,
that $\varpi^{NA} h(t_1, \dots, t_n) r^m \in R_0$
for every $h \in R_0[T_1, \dots, T_n]$ that is homogeneous of degree $N$ and
all~$m \geq 0$.
The case $m=0$ is clear because $\varpi^At_i\in R_0$ for all~$i$.
Induction step: $m>0$.
It is sufficient to consider the case where $h$ is a monomial, $h=T_1^{e_1} \dots T_n^{e_n}$.
Since $e_1+\dots+e_n=N=d_1+\dots+d_n$, there is at least one $i$ for which $e_i \geq d_i$.
Without loss of generality we can assume that $i=1$.
Now
$\varpi^{NA} t_1^{e_1} \dots t_n^{e_n} r^m = \varpi^{NA} t_1^{e_1-d_1} t_2^{e_2} \dots t_n^{e_n} g_1(t_1, \dots, t_n) r^{m-1}$
and by the induction hypothesis this is in~$R_0$.
This concludes the induction proof.

There exist $a_1, \dots, a_n \in R$ such that $a_1 t_1+\dots+a_n t_n=1$.
Take $B\geq0$ such that $\varpi^Ba_i\in R_0$ for all~$i$.
Applying the above result to $h = (\varpi^Ba_1 T_1+\dots+\varpi^Ba_n T_n)^N$
shows that $\varpi^{N(A+B)}r^m \in R_0$ for all~$m \geq 0$, and hence
$r\in R^\circ$.
\end{proof}

\begin{corollary}\label{uniformimpliesstrict}
Let $(R,R^+)$ be a uniform Tate affinoid ring,
and let $X=\Spa(R,R^+)$ be the associated affinoid pre-adic space. Let $t\in R$
and set $U=\{x\in X:|t(x)|\leq 1\}$ and $V=\{x\in X:|t(x)|\geq1\}$.
Then the sequence
$$0\to\OX(X)\to\OX(U)\oplus\OX(V)\to\OX(U\cap V)\to 0$$
is exact.
\end{corollary}

\begin{proof}
The conclusion of the corollary is condition~(ii)
of Lemma~\ref{exact}, so it suffices to verify condition~(iii) of
that lemma.
Applying Lemma~\ref{alain} with $t_1=1$ and $t_2=t$ ($\phi_1 $ is the
identity, $\phi_2=\phi$) we deduce $A_0\cap\phi^{-1}(B_0)\subseteq R^\circ$,
and we can conclude because $R^\circ\subseteq\varpi^{-n}R_0$ for some
$n\geq0$ by uniformity.
\end{proof}

\begin{corollary}\label{locallyzeroimpliestopnilp}
If $X$ is a Tate affinoid pre-adic space, $f\in\OX(X)$,
and $X$ has a cover by opens $U_i$ such that $f|U_i=0$ for all $i$,
then $f$ is topologically nilpotent.
\end{corollary}

\begin{proof}
By Lemma~\ref{covers}(i) we may assume the cover is rational.
By Lemma~\ref{alain} any locally zero element is power-bounded.
Applying this to $\varpi^{-1}f$ we see that $\varpi^{-1}f$ is power-bounded
and hence $f$ is topologically nilpotent.
\end{proof}

\begin{remark}
We will need Lemma~\ref{alain} and Corollary~\ref{uniformimpliesstrict} later,
but
Peter Scholze points out to us that Corollary~\ref{locallyzeroimpliestopnilp}
also follows easily from Theorem~1.3.1 of~\cite{berkovich:book}.
\end{remark}

We now give a new criterion for the presheaf $\OX$ on $\Spa(R,R^+)$
to be a sheaf. Let us say that a Tate affinoid ring $(R,R^+)$
(or the associated pre-adic space $\Spa(R,R^+)$) is
\emph{stably uniform}
if every rational subset $U\subseteq\Spa(R,R^+)$ has the property
that $\OX(U)$ is uniform.
We remark that a Tate ring $R$ is uniform
iff its completion is, by Lemma~\ref{basic}(iii).

\begin{theorem}\label{stablyuniformimpliessheaf}
Let $(R,R^+)$ be a stably uniform Tate affinoid ring.
Then $X:=\Spa(R,R^+)$ is an adic space, in other words, the presheaf
$\OX$ on $X$ is a sheaf of complete topological rings.
\end{theorem}

Note that there are no finiteness hypotheses on~$R$ whatsoever. 
Before we embark upon the proof, let us remark that its deduction
from Corollary~\ref{uniformimpliesstrict} is, to a large
extent, an application of standard machinery, although unfortunately
we have found no single reference in the literature that fully covers
our requirements. The following sources were of great use to us:
\S2 of Huber's paper~\cite{huber:generalization} (proving an
analogous result for adic spaces under some Noetherian hypotheses),
Chapter~8 of~\cite{bgr} (proving Tate's acyclicity theorem for affinoid
$k$-algebras topologically of finite type) and finally \S8.2 of~\cite{wedhorn}.
As preparation we now consider two special types of
covers of affinoid pre-adic spaces and some relationships between them.

Say $(R,R^+)$ is a Tate affinoid ring,
and $t_1,t_2,\ldots,t_n\in R$ are elements of~$R$
such that the ideal they generate is all of~$R$.
Set $X=\Spa(R,R^+)$, and for $1\leq i\leq n$
define $U_i:=\{x\in X\,:\,|t_j(x)|\leq |t_i(x)|\,\mbox{for all $1\leq j\leq n$}\}$.
Then each $U_i$ is a rational subset of~$X$ and, because the $t_i$'s generate 1,
the union of the $U_i$ is~$X$. Such a cover is called a \emph{rational cover}.
If furthermore each $t_i\in R^\times$, the cover is called
a \emph{rational cover generated by units}.

Now say $(R,R^+)$ is a Tate affinoid ring, and
$t_1,t_2,\ldots,t_n\in R$. Set $X=\Spa(R,R^+)$, and for each subset
$I$ of $\{1,2,\ldots,n\}$ define
$U_I=\{x\in X\,:\,|t_i(x)|\leq1\mbox{\ for\ }i\in I,\,|t_i(x)|\geq1\mbox{\ for\ }i\not\in I\}$.
Then each $U_I$ is a rational subset of~$X$ and the union of the
$2^n$ sets $U_I$ is~$X$.  Such a cover is called a \emph{Laurent
cover}.

\begin{lemma}[Huber]\label{covers}
Let $X$ be a Tate affinoid pre-adic space.
\begin{enumerate}[(i)]
\item For every open cover $\mathcal U$ of $X$, there exists a rational
cover $\mathcal V$ of $X$ which is a refinement of~$\mathcal U$.
\item For every rational cover $\mathcal U$ of $X$, there exists a
Laurent cover $\mathcal V$ of $X$ such that for every $V\in\mathcal V$,
the cover $\{U\cap V\,:\,U\in\mathcal U\}$ of $V$ is a rational cover
generated by units.
\item For every rational cover $\mathcal U$ of $X$ generated by units,
there exists a Laurent cover $\mathcal V$ of $X$ which is a refinement
of~$\mathcal U$.
\end{enumerate}
\end{lemma}

\begin{proof}
(i) See~\cite{huber:generalization} Lemma 2.6.

(ii) If $\mathcal U$ is generated by $t_1,t_2,\ldots,t_n\in R$ 
then by assumption there are $a_i\in R$ such that $\sum_i a_i t_i=1$.
Because $R^+$ is open there exists $B\geq0$ such that $\varpi^Ba_i\in R^+$ for
all~$i$. For all $x\in X$ we have by definition that $|r(x)|\leq 1$
for all $r\in R^+$,
so from $\varpi^B=\sum_i(\varpi^Ba_i)t_i$ it follows that
$|\varpi^B(x)|\leq\max_i|t_i(x)|$, so (since $|\varpi(x)|<1$ by continuity of~$x$)
$|\varpi^{B+1}(x)|<\max_i|t_i(x)|$.
One checks easily, see for example the proof of~\cite{bgr} Lemma~8.2.2/3,
that the Laurent cover generated by the $ct_i$ with $c=\varpi^{-(B+1)}$
has the desired property.

(iii) This can be shown by the purely combinatorial argument in the
proof of~\cite{bgr} Lemma 8.2.2/4.
\end{proof}

\begin{proof}[Proof of Theorem~\ref{stablyuniformimpliessheaf}]
Note
that a rational subset of a stably uniform
affinoid pre-adic space is again stably uniform.
First consider $\OX$ as a presheaf of abelian groups on~$X$.
We claim that any Laurent cover of any rational subset of $X$
is $\OX$-acyclic. We prove this by induction on~$n$, the number
of functions $t_i$ used to define the Laurent cover.
For $n=1$ the claim is just Corollary~\ref{uniformimpliesstrict}
and the inductive step is proved following \cite{bgr} Corollary~8.1.4/4. 

The proof of \cite{bgr} Proposition~8.2.2/5, using Lemma~\ref{covers}(i)--(iii)
in lieu of \cite{bgr} Lemmas~8.2.2/2--4, now shows that any cover by
rational subsets of any rational subset of $X$ is $\OX$-acyclic.
It follows that $\OX$ is a sheaf of abelian groups on
the site whose objects are rational subsets of $\Spa(R,R^+)$
and whose covers are covers of rational subsets by rational
subsets. 

Since $\OX$ is a presheaf of rings and a sheaf of abelian groups, it
is also a sheaf of rings on this site. We claim that it is even a sheaf
of complete topological rings on this site.  
For this it suffices,
by the first paragraph of \S2 of~\cite{huber:generalization}, to check
that if $U=\cup_i U_i$ is a cover of a rational subset $U$ by rational
subsets, then the induced map $\OX(U)\to\prod_i\OX(U_i)$ is strict.
By Lemma~\ref{covers}(i) there is a rational cover $U=\cup_j V_j$
that refines $U=\cup_i U_i$. Lemma~\ref{alain} and the uniformity
hypothesis imply that the induced
map $\OX(U)\to\prod_j\OX(V_j)$ is strict and since this map
factors through $\OX(U)\to\prod_i\OX(U_i)$ that map must be strict
too.

We have established that $\OX$ is a sheaf of complete topological rings
on the basis of rational subsets of~$X$ and by~\cite{ega1} Chapter~0 (3.2.2)
we deduce that $\OX$ is a sheaf of complete topological rings on~$X$.
\end{proof}

As a toy example of an application, we get a new proof
of Theorem~6.3(iii) of~\cite{scholze:perfectoid}, that
avoids the arguments of 6.10--6.14 of~{\it loc.\ cit.}

\begin{corollary}\label{perfectoid}
If~$k$ is a perfectoid field then the affinoid pre-adic space associated to
a perfectoid $k$-algebra is an adic space.
\end{corollary}

\begin{proof}
Perfectoid affinoid $k$-algebras are uniform (by definition) and hence
stably uniform (by Corollary~6.8 of~\cite{scholze:perfectoid}),
so the theorem directly implies that the affinoid pre-adic space associated
to a perfectoid affinoid $k$-algebra is an adic space.
\end{proof}

\begin{remarkn} Brian Conrad notes that Scholze has systematically removed
all need for a ground field in his perfectoid theory, in his 2014 UC Berkeley
course, and in particular apparently Corollary~6.8 is valid more generally
with a uniform and Tate condition. Hence our arguments will show that the
affinoid pre-adic space associated to a perfectoid ring is an adic space.
A similar comment applies to the following corollary.
\end{remarkn}

We also deduce that under the stably uniform assumption, in characteristic~$p$
we can check that a ring is perfectoid locally.

\begin{corollary}\label{216incharp} If~$k$ is a perfectoid field of characteristic~$p>0$, and if~$A$ is a stably uniform complete Tate $k$-algebra
such that $\Spa(A,A^+)$ has a rational cover by affinoids of the form
$\Spa(R_i,R_i^+)$ with the $R_i$ perfectoid $k$-algebras, then $A$
is perfectoid.
\end{corollary}

\begin{proof} 
By Proposition~5.9 of~\cite{scholze:perfectoid}, it suffices to
show that the $p$th power map $A\to A$ is surjective. So say $a\in A$.
Let $a_i$ denote the restriction of $a$ to $R_i$; then because $R_i$
is perfectoid (and hence reduced) we know $a_i=(b_i)^p$ for a unique
$b_i\in R_i$. A rational subspace of an affinoid perfectoid space
is again perfectoid, by Theorem~6.3 of~\cite{scholze:perfectoid},
and hence the $b_i$ agree on overlaps; Theorem~\ref{stablyuniformimpliessheaf}
implies that the $b_i$ glue together to give an element $b\in A$.
Now $b^p-a$ is locally zero and hence zero (again by Theorem~\ref{stablyuniformimpliessheaf}), and hence $b^p=a$.
\end{proof}

\section{Counterexamples}\label{counterexamples}

Throughout this section, $k$ is a field complete with
respect to a non-trivial non-archimedean norm $|.|:k\to\R_{\geq0}$, $\Ok$ is
its integer ring, $\varpi$ is an element of $k^\times$ with $0<|\varpi|<1$,
$R$ will be a $k$-algebra and $R_0$ will be an $\Ok$-subalgebra of $R$ such
that~$kR_0=R$. We call such an~$R$ (equipped with the topology coming
from $R_0$ and $\varpi$) a Tate $k$-algebra; such~$R$ are Tate rings.

In this section we give various examples of affinoid $k$-algebras
for which the structure presheaf is not a sheaf of complete
topological rings (and is not even
a sheaf of abelian groups).
Let us say that an affinoid $k$-algebra $(R,R^+)$ is \emph{sheafy}
if $X:=\Spa(R,R^+)$ is an adic space
(that is, if $\OX$ is a sheaf of complete topological rings).
We remark here that as this paper
was being written, a preprint of Tomoki Mihara appeared
on the ArXiv~\cite{mihara} with another example; Mihara's work
was independent of ours.

The following lemma will be helpful for us when attempting to locate
the power-bounded elements in polynomial rings (which are naturally
graded).

\begin{lemma}\label{grading}
Let $R$ be a
Tate $k$-algebra with topology
defined by an $\Ok$-subalgebra $R_0$. Say we are given
a torsion-free (additive) abelian group~$G$ and a $G$-grading of~$R$, that is,
a decomposition $R=\bigoplus_{g\in G} R^{(g)}$
where the $R^{(g)}$ are $k$-subspaces
of $R$ satisfying $R^{(g)}R^{(h)}\subseteq R^{(g+h)}$. Suppose that $R_0$
is also graded by this grading, that is, $R_0=\bigoplus_{g\in G}(R_0)^{(g)}$,
with $(R_0)^{(g)}=R_0\cap R^{(g)}$.
Then $R^\circ$ is also graded by this grading.
\end{lemma}

\begin{proof}
Say $r\in R^\circ$. Then $r=\sum_{i\in I}r_i$, with $I\subseteq G$
a finite subset and $r_i\in R^{(i)}$. It suffices to check that $r_i\in R^\circ$
for all $i$. We do this by induction on the size of $I$. If $|I|\leq1$
the result is clear. For $|I|>1$ we let $H$ be the subgroup of~$G$
generated by $I$ and observe that $H$ is finitely-generated and torsion-free,
and hence a free abelian group, so there is an injection $H\to\R$,
giving us an ordering on~$I$. Say $i_0$ is the smallest element of~$I$
with respect to this embedding. Write $r=r_0+r_1$ with $r_0=r_{i_0}$.
Because $r\in R^\circ$ there is some $N$ such that $r^n\in\varpi^{-N}R_0$
for all $n\geq0$, and hence $r_0^n+r_2\in\varpi^{-N}R_0$, where
$r_0^n\in R^{(ni_0)}$ and $r_2$ is a sum of elements in $R^{(j)}$ for $j\in H$,
$j>ni_0$. In particular $r_0^n$ must be in $\varpi^{-N}(R_0)^{(ni_0)}$
and in particular $r_0^n\in\varpi^{-N}R_0$, hence $r_0\in R^\circ$,
and hence $r_1\in R^\circ$ and we can apply the inductive hypothesis to $r_1$,
finishing the argument.
\end{proof}

\subsection{A finitely-generated non-sheafy $k$-algebra}

Even if $R$ is a Tate $k$-algebra which
is finitely-generated as an abstract $k$-algebra,
the subalgebra $R_0$ defining the topology might be sufficiently
nasty to ensure that $(R,R^+)$ is not sheafy. This is not
surprising -- indeed Rost's example of a non-sheafy ring (which
is not a $k$-algebra) given at the end of \S1 of~\cite{huber:generalization}
is finitely-generated over~$\Z$. We remark here that before~\cite{mihara},
Rost's example was the only example known to us
in the literature of a non-sheafy ring.
The key idea of the following counterexample is basically Rost's.

Now, let $R$ be the ring $k[T,T^{-1},Z]/(Z^2)$
and let 
$R_0$ denote the $\Ok$-submodule of~$R$ with $\Ok$-basis 
$\varpi^{|n|}T^n$ and $\varpi^{-|n|}T^nZ$ ($n\in\Z$).
(For the avoidance of doubt, here $|\cdot|$ denotes the ordinary
absolute value on $\Z$.)
One checks easily that $R_0$ is an $\Ok$-subalgebra of $R$ and that $kR_0=R$.
We note in passing that $R_0$ is
not Noetherian -- indeed, the ideal $ZR\cap R_0$ of $R_0$ is easily
checked to be not finitely-generated. However, $Z\in R_0$ is nilpotent
and $R_0/(Z)$ is Noetherian.

\begin{proposition}
For the space $X:=\Spa(R,R^\circ)$ the presheaf $\OX$
is not a sheaf. In particular, $X$ is covered by $U:=\{x\in X:|T(x)|\leq 1\}$
and $V:=\{x\in X:|T(x)|\geq1\}$ and the map $\OX(X)\to\OX(U)\oplus\OX(V)$
is not injective.
\end{proposition}

Before we begin the proof, we briefly
note two consequences. Firstly this proposition (positively)
resolves the footnote just before Definition~2.16 in~\cite{scholze:perfectoid}.
Secondly, $R$ is Noetherian, but it cannot be strongly
Noetherian because $\OX$ is a sheaf for strongly Noetherian Tate
$k$-algebras by Theorem~2.2 of~\cite{huber:generalization}.

\begin{proof}
That $X$ is covered by the opens~$U$ and $V$ is obvious. By definition,
$\OX(X)=\varprojlim_n R/\varpi^nR_0$, the completion of~$R$. Similarly,
$\OX(U)=\varprojlim_n R/\varpi^nR_0[T]$ and
$\OX(V)=\varprojlim_n R/\varpi^nR_0[T^{-1}]$. We claim that the
map $\OX(X)\to\OX(U)\oplus\OX(V)$ is not injective, and this
suffices to show that $\OX$ is not even a sheaf of abelian groups
on~$X$. 

More precisely, we claim that $0\not=Z\in\OX(X)$ but that $Z$
restricts to zero in both~$U$ and~$V$. To verify the first assertion
it suffices to observe that $kZ\not\subseteq R_0$, which is
clear because $kZ\cap R_0=\Ok Z$. To verify the second assertion it suffices
to check that $kZ\subseteq R_0[T]$ and $kZ\subseteq R_0[T^{-1}]$;
but both of these are also clear because for $n\geq0$ we have
$\varpi^{-n}Z=\varpi^{-n}T^{-n}Z.T^n\in R_0[T]$ and
$\varpi^{-n}Z=\varpi^{-n}T^nZ.T^{-n}\in R_0[T^{-1}]$.
\end{proof}

Note that $\OX(U)$ is the completion of $k[T,T^{-1}]$ with
respect to the topology generated by the subring $\Ok[T,\varpi T^{-1}]$
so in fact $U$ is isomorphic to the adic space associated to the
annulus $\{|\varpi|\leq |T|\leq 1\}$. Similarly $V$ is isomorphic
to the adic space associated to the annulus $\{1\leq |T|\leq|\varpi|^{-1}\}$;
however, $X$ is not the adic space associated to the annulus
$\{|\varpi|\leq |T|\leq |\varpi|^{-1}\}$ as $\OX(X)$ contains nilpotents.

\subsection{A non-perfectoid, locally perfectoid space.}

In this subsection we assume the characteristic of~$k$ is~$p$,
and that $k$ is perfectoid (or equivalently that $k$ is perfect).
In this situation we can basically ``perfectify'' our previous example,
and in this way construct an affinoid pre-adic space which is
not adic (and in particular not perfectoid), but which is locally
perfectoid. In particular we resolve Conjecture~2.16 of~\cite{scholze:survey}
(negatively).

The details are as follows. We start by perfectifying the ring $k[T,T^{-1}]$,
that is, we take the direct limit $\varinjlim_{x\mapsto x^p}k[T,T^{-1}]$;
we call this ring $k[T^{1/p^{\infty}},T^{-1/p^\infty}]$. We then adjoin
a nilpotent by setting $R=k[T^{1/p^\infty},T^{-1/p^\infty}][Z]/(Z^2)$.
Then $R$ has a $k$-basis consisting of elements of the form $T^n$
and $T^nZ$ for $n\in\Z[1/p]$. We let $R_0$ denote the $\Ok$-submodule
of~$R$ with basis $\varpi^{|n|}T^n$ and $\varpi^{-|n|}T^nZ$ ($n\in\Z[1/p]$).
Topologize~$R$ as usual
by letting subsets of the form $r+aR_0$ ($r\in R, a\in k^\times$)
be a basis.

\begin{proposition}
The space $X:=\Spa(R,R^\circ)$ is not an adic space,
because $\OX$ is not a sheaf. However $X=U\cup V$ with
$U:=\{x\in X\,:\,|T(x)|\leq 1\}$ and $V:=\{x\in X:\,|\,T(x)|\geq1\}$
both perfectoid spaces.
\end{proposition}

\begin{proof}
We have $\varpi^{-1}Z\not\in R_0$ and hence
$Z\not=0$ in $\OX(X)$. But as before $kZ\subseteq R_0[T]$
and $kZ\subseteq R_0[T^{-1}]$, and hence $Z$ restricts to zero
on both~$U$ and~$V$, so again $\OX$ is not a sheaf. 

Next observe that the completion of $R$ with respect to the
basis given by $r+aR_0[T]$, $r\in R$, $a\in k^\times$, is
equal to the completion of $k[T^{1/p^\infty},T^{-1/p^\infty}]$
with respect to the topology defined by the subring
$\Ok[T^{1/p^\infty},(\varpi/T)^{1/p^\infty}]$ (that is, the
direct limit of $\Ok[T,\varpi/T]$ via $x\mapsto x^p$);
from this we deduce that $U$ is the $p$-finite affinoid perfectoid space
associated to the annulus $\{|\varpi|\leq |T|\leq 1\}$; similarly $V$
is perfectoid. 
\end{proof}

Scholze (personal communication) observes that $R^\circ$ in the lemma
above is not bounded (as it contains the line $kZ$) and asks
whether his Conjecture~2.16 becomes true under the additional assumption
that the ring is uniform. Explicitly, if $A$ is uniform and complete,
and $\Spa(A,A^+)$ has a cover by rational subsets which are perfectoid,
is~$A$ perfectoid? One might also ask whether the conjecture becomes true
if~$A$ is assumed stably uniform, where the question becomes more
accessible -- indeed we resolved this in the characteristic~$p$
case in Corollary~\ref{216incharp}, and perhaps minor modifications
of these arguments will also deal with the general case.

\subsection{An affinoid pre-adic space with a non-nilpotent locally zero element.}

We have seen examples of global sections of affinoid pre-adic $k$-spaces
which are non-zero but locally zero. The examples we have seen
so far were nilpotent, which is perhaps not surprising:
by Corollary~\ref{locallyzeroimpliestopnilp} any such example
has to be topologically nilpotent. Here we give an example
of a section which is locally zero but genuinely not nilpotent.

Set $R=k[T,T^{-1},Z]$ and let $R_0$ be the $\Ok$-subalgebra generated
by $\varpi T$, $\varpi T^{-1}$, and for $n\geq1$ the elements
$\varpi^{-n}T^{a(n)}Z$ and $\varpi^{-n}T^{-b(n)}Z$,
where $a(n)$ and $b(n)$ are two sequences
of positive integers both tending to infinity rapidly. More
precisely, the following will suffice: set $a(1)=1$ and
then for $J\geq1$ ensure that
$$b(J)>J^2+J\max\{b(j):1\leq j<J,\;a(i):1\leq i\leq J\}$$
and for $I\geq2$ ensure that
$$a(I)>I^2+I\max\{b(j):1\leq j<I,\;a(i):1\leq i<I\}.$$
The sequence $a(1),b(1),a(2),b(2),\ldots$ can be constructed
recursively such that these inequalities are satisfied.

\begin{proposition}
Let $X=\Spa(R,R^\circ)$. Then $Z\in\OX(X)$ is not nilpotent
but vanishes on the subsets $U:=\{x\,:\,|T(x)|\leq1\}$ and
$V:=\{x\,:\,|T(x)|\geq1\}$ that cover~$X$.
\end{proposition}

\begin{proof}
By construction $\varpi^{-n}Z\in R_0[T]$ and
$\varpi^{-n}Z\in R_0[T^{-1}]$ for all $n\geq1$, so $Z$ is is zero on
$U$ and~$V$. To see that $Z$ is not
nilpotent on $\Spa(R,R^\circ)$ we need to show that $Z^e$ is non-zero
in the completion $\Rhat$ of $R$ for any $e\geq1$,
so we need to verify that for all
$e\geq1$ there exists
some $M(e)\geq0$ such that $\varpi^{-M(e)}Z^e\not\in R_0$.

The ring $R_0$ is graded by $\Z\times\Z$ (the powers of $T$ and $Z$)
and the given generators of $R_0$ are homogeneous. It follows from this
that if $\varpi^{-M}Z^e\in R_0$ then
$\varpi^{-M}Z^e$ will be an $\Ok$-linear sum of products of the
given generators, where each of these products
is of the form $\lambda Z^e$ (with $\lambda\in k^\times$).
So it suffices to check that for any
$e\geq1$ there
exists some bound
$M(e)\geq0$ such that if $\lambda Z^e$ is a product
of the given generators of $R_0$ then $|\lambda|<|\varpi|^{-M(e)}$.

Set $\alpha_n=\varpi^{-n}T^{a(n)}Z$ and $\beta_n=\varpi^{-n}T^{-b(n)}Z$.
Say $\lambda Z^e$ is a product of the given generators of $R_0$, and
let us consider which $\alpha_i$ and $\beta_j$ occur in this product.
There are two
cases. If the product mentions only $\alpha_i$ and $\beta_j$ for $i,j\leq e$,
then (because of the coefficient of $Z$) the product can mention only $e$ such
elements, so
$|\lambda|\leq|\varpi^{-e^2}|$. If, however, the product mentions some $\alpha_i$
or $\beta_i$ with $i>e$ then we claim that $|\lambda|\leq 1$, and it
suffices to prove this claim.
Let $I$ denote the largest $i$ such that $\alpha_i$ is mentioned
(with $I=0$ if no $\alpha_i$ are mentioned),
and let $J$ denote the largest $j$ such that $\beta_j$ is mentioned
(with $J=0$ if no $\beta_j$ is mentioned). 
Write $\lambda Z^e=(\varpi T)^\ell(\varpi T^{-1})^m (\varpi^{-\mu} T^\nu Z^e)$,
with $\varpi^{-\mu}T^{\nu}Z^e$ a product of $\alpha$s and $\beta$s.
If $I\leq J$ then because $b(J)>J^2+J\max\{b(j):j<J,\;a(i):i\leq J\}$
we see that $|\nu|\geq b(J)-(e-1)\max\{b(j):j<J,\;a(i):i\leq J\}>J^2$
whereas $\mu\leq eJ<J^2$, and because one of $\ell$ and $m$
must be at least as big as $|\nu|$ to kill the power of~$T$,
we see $|\lambda|\leq 1$. A similar argument works in the case $I>J$,
this time using the defining property of $a(I)$.
\end{proof}

\subsection{Exactness failing in the middle.}

Scholze (personal communication) asked whether one could construct
an example of an affinoid pre-adic $k$-space for which $\OX$
fails to be a sheaf for a reason other than the existence of
sections which are locally zero but non-zero. Here is such
an example -- an example where glueing fails.

If $R$ is a Tate $k$-algebra that contains a $k$-basis $\{r_1,r_2,\ldots\}$
and there exist non-negative integers $\{n_1,n_2,\ldots\}$ with the
property that the $r_i$ are an $\Ok$-basis for $R^\circ$, and $\varpi^{n_i}r_i$
are an $\Ok$-basis for $R_0$, then $(\Rhat)^\circ$ contains no line,
because $(\Rhat)^\circ=\widehat{(R^\circ)}$ by Lemma~\ref{basic}(iii).
So Corollary~\ref{locallyzeroimpliestopnilp} implies
that if $X=\Spa(R,R^\circ)=U\cup V$ then the map $\OX(X)\to\OX(U)\oplus\OX(V)$
must be injective (as the kernel is a $k$-vector space all of whose
elements are power-bounded). Here however is an example where
$\OX(X)\to\OX(U)\oplus\OX(V)\to\OX(U\cap V)$ is not exact -- there
are global sections of $U$ and~$V$ which agree on $U\cap V$ but
which do not glue together to give a section on $U\cup V$. 

Set $R=k[T,T^{-1},Z_1,Z_2,\ldots]$ and let $R_0$ be the free $\Ok$-submodule
of $R$ generated by elements $\varpi^d T^a \prod_i Z_i^{e_i}$ with
$d,a\in\Z$ and $e_1,e_2,\ldots\in\Z_{\geq0}$ satisfying
\begin{enumerate}[(i)]
\item if $\sum_i e_i=0$ then $d=|a|$;
\item if $\sum_i e_i=1$ then $d=|a|-2\min\{\sum_i ie_i,|a|\}$;
\item if $\sum_i e_i\ge2$ then $d=|a|-2\sum_i ie_i$.
\end{enumerate}
It is easily checked that the product of two such generators is in $R_0$ and
that $1\in R_0$, so $R_0$ is a ring. It is clear that $kR_0=R$.
Note that $\varpi T$ and $\varpi T^{-1}$ are in $R_0$ but not in $\varpi R_0$.
Set $U=\{x\in X\,:\,|T(x)|\leq1\}$ and $V=\{x\in X\,:\,|T(x)|\geq1\}$ as usual;
set $A=B=R$ and topologize them using $A_0=R_0[T]$ and $B_0=R_0[T^{-1}]$,
so $\OX(U)=\Ahat$ and $\OX(V)=\Bhat$.

\begin{proposition}
$\widehat{(R^\circ)}$ contains no line.
\end{proposition}

\begin{proof}
By Lemma~\ref{grading}, $R^\circ$ is graded by the degrees of $T,Z_1,Z_2,\ldots$.
It is easy to check that
$\varpi^d T^a \prod_i Z_i^{e_i}$ is in $R^\circ$ iff
\begin{enumerate}[(i)]
\item if $\sum_i e_i=0$ then $d\ge|a|$;
\item if $\sum_i e_i\geq1$ then $d\ge|a|-2\sum_i ie_i$.
\end{enumerate}
This, together with the arguments above, shows that $\widehat{(R^\circ)}$
contains no line.
\end{proof}

Note that for all $n\geq1$ we have
$\varpi^{-n}Z_n=\varpi^{-n}T^{-n}Z_n.T^n\in R_0[T]$ and similarly
$\varpi^{-n}Z_n\in R_0[T^{-1}]$ but $\varpi^{-1}Z_n\not\in R_0$, so
$\varpi^{n-1}\left(R_0[T]\cap R_0[T^{-1}]\right)\not\subseteq R_0$ for every
$n\geq1$ and hence the map $R\to A\oplus B$ is not strict.

Now because $Z_n\in\varpi^{n}A_0$ we have that $\sum_i Z_i$ converges in
$\Ahat$; let $a$ be the limit. Similarly it converges in $\Bhat$;
let $b$ be the limit.

\begin{proposition}
$a\in\Ahat$ and $b\in\Bhat$ agree on $U\cap V$, but cannot be glued to
an element of $\Rhat$.
\end{proposition}

\begin{proof}
That $a$ and $b$ agree on $U\cap V$ is obvious, because the image of $a$
in $\Chat$ and the image of $b$ in $\Chat$ both are the limit of $\sum_i Z_i$
in~$\Chat$.

Let $r\in\Rhat$. There is a Cauchy sequence $r_1,r_2,\ldots$ in $R$ with limit~$r$.
For each $n\geq1$, let $\rho_n:R\to k$ be the map that sends an element of $R$ to
the coefficient of $Z_n$ of its $Z_n$ graded piece. This map is continuous and
therefore factors through a unique continuous map $\widehat{\rho}_n:\Rhat\to k$.
Similarly we define $\widehat{\alpha}_n:\Ahat\to k$.

We claim that $\widehat{\rho}_1(r),\widehat{\rho}_2(r),\ldots$ converges to zero.
Let $M\geq0$. There exists $I\geq1$ such that $r_i-r_j\in\varpi^MR_0$ for all $i,j\geq I$.
It follows that $\rho_n(r_i)-\rho_n(r_j)\in\varpi^M\Ok$  for all $i,j\geq I$ and all
$n\geq1$. Take $N\geq1$ such that none of $Z_N,Z_{N+1},\ldots$ occurs in $r_I$.
For all $n\geq N$ we have $\rho_n(r_I)=0$, so $\rho_n(r_i)\in\varpi^M\Ok$ for all
$i\geq I$, so $\widehat{\rho}_n(r)\in\varpi^M\Ok$. This concludes the proof that
$\widehat{\rho}_n(r)$ converges to zero.

It is easily seen that $\widehat{\alpha}_n(a)=1$ for all $n\geq1$. Since
$\widehat{\rho}_n$ and $\widehat{\alpha}_n$ are compatible through
$\Rhat\to\Ahat$, it follows that the image in $\Ahat$ of $r$ cannot be~$a$.
\end{proof}

\subsection{A uniform space with a subspace containing a line of power-bounded elements.}

We now give an example of a uniform space that is not stably uniform.
See also \cite{mihara}, which was written independently.

Consider the free $\Ok$-submodule $R_0$ of $k[T,T^{-1},Z]$ generated
by $(\varpi T)^a(\varpi Z)^b$ with $b\geq0$ and $a\geq-b^2$.
It is easily verified that $R_0$ is also an $\Ok$-subalgebra; indeed
if $a\geq-b^2$ and $a'\geq-(b')^2$ then
$a+a'\geq-b^2-(b')^2\geq-(b+b')^2$ if $b,b'\geq0$. Set $R=A=kR_0$
and topologize them using $R_0\subseteq R$ and $A_0=R_0[T]\subseteq A$.

\begin{proposition}
The affinoid $k$-algebra $(R,R^\circ)$ is uniform, but not stably uniform.
More specifically, $A^\circ$ contains the non-zero line~$kZ$.
\end{proposition}

\begin{proof}
We claim that $R^\circ=R_0$. By Lemma~\ref{grading} it suffices to
check that for every $r=(\varpi T)^a(\varpi Z)^b\in R_0$
(with $b\geq0$ and $a\geq-b^2$) and $\lambda\in k$, if $\lambda r\in R^\circ$
then $\lambda\in\Ok$. An elementary calculation shows that it
then suffices to check that $\varpi^{-1}r^n\not\in R_0$ for any $n\geq1$,
and this is easily checked.

Note that $T^{-1}Z=(\varpi T)^{-1}(\varpi Z)\in R_0$ and hence $Z\in A_0$,
but $\varpi^{-1}Z\not\in A_0$ (this is not hard to see, using the
grading on $A_0$). 
However, for $n\geq1$ we have
$(\varpi^{-n}Z)^{n+1}=(\varpi T)^{-n^2-2n-1}\allowbreak(\varpi Z)^{n+1}\allowbreak T^{n^2+2n+1}\in A_0$
and hence $\varpi^{-n}Z\in A^\circ$ for all $n\geq1$.
\end{proof}

\subsection{A uniform affinoid space which is non-sheafy.}

Finally we give an example of a uniform affinoid $X$ over $k$
for which $\OX$ is not a sheaf. Let $R_0$ be the free $\Ok$-submodule
of $k[P,P^{-1},\allowbreak Q,Q^{-1},\allowbreak T,T^{-1},\allowbreak Z]$
generated by elements $\varpi^dP^pQ^qT^aZ^e$ with $d,p,q,a\in\Z$
and $e\in\Z_{\geq0}$ satisfying the following conditions:
\begin{enumerate}[(i)]
\item $d=\max\{p+q+a,p+q-a,p+a,q-a\}$;
\item if $e=0$ then $p\geq0$ and $q\geq0$;
\item if $e=1$ then $p\geq0$ or $q\geq0$.
\end{enumerate}
If $\varpi^{d_1}P^{p_1}Q^{q_1}T^{a_1}Z^{e_1}$ and $\varpi^{d_2}P^{p_2}Q^{q_2}T^{a_2}Z^{e_2}$
are two such elements, then $d_1\geq p_1+q_1+a_1$ and $d_2\geq p_2+q_2+a_2$ and hence $d_1+d_2\geq (p_1+p_2)+(q_1+q_2)+(a_1+a_2)$ and so on; from this it is not hard to see that the product of two $\Ok$-module generators
of $R_0$ is in $R_0$; moreover $1\in R_0$, and hence $R_0$ is a ring. Set $R=kR_0$.
\begin{proposition}
The affinoid $k$-algebra $(R,R^\circ)$ is uniform, but for the space $X:=\Spa(R,R^\circ)$
the presheaf $\OX$ is not  a sheaf.  In particular, $Z$ is non-zero on the subspace
$W:=\{x\in X\,:\,|P(x)|\leq 1\mbox{\rm\ and }|Q(x)|\leq1\}$ but vanishes on the
subspaces $U:=\{w\in W\,:\,|T(w)|\leq 1\}$ and $V:=\{w\in W\,:\,|T(w)|\geq 1\}$
that cover~$W$.
\end{proposition}

\begin{proof}
One verifies using Lemma~\ref{grading} that $R^\circ=R_0$. Thus, $R$ is uniform.

The subspace $W$ has global sections given by the completion of the ring $A=R$
with respect to the topology defined by $A_0=R_0[P,Q]$.
Now we have $\varpi^{-n}Q^{-2n}T^{-n}Z\in R_0$, so
$\varpi^{-n}Z=\varpi^{-n}Q^{-2n}T^{-n}Z.Q^{2n}T^n\in A_0[T]$,
for all $n\geq0$.
Similarly $\varpi^{-n}Z\in A_0[T^{-1}]$ for all $n\geq0$ and we deduce that
$Z$ vanishes on the subspaces $U$ and~$V$.
However, $\varpi^{-1}P^{-m}Q^{-n}Z$ is not in $R_0$ for any $m,n\geq0$
(indeed for $m,n>0$ this is not even in $R$), so $\varpi^{-1}Z$ is not
in $A_0=R_0[P,Q]$ and so $Z$ is a non-zero function on~$W$.
\end{proof}

\bibliographystyle{amsalpha}
\providecommand{\bysame}{\leavevmode\hbox to3em{\hrulefill}\thinspace}
\providecommand{\MR}{\relax\ifhmode\unskip\space\fi MR }
% \MRhref is called by the amsart/book/proc definition of \MR.
\providecommand{\MRhref}[2]{%
  \href{http://www.ams.org/mathscinet-getitem?mr=#1}{#2}
}
\providecommand{\href}[2]{#2}

\vskip\baselineskip

\end{document}